\documentclass[12pt]{amsart}

\usepackage{amssymb,amsfonts}
\usepackage{amsthm}
\usepackage[dvipsnames]{xcolor}
\usepackage{graphicx}
\usepackage[all,arc]{xy}
\usepackage[colorlinks=true, allcolors=blue]{hyperref}
\usepackage{enumerate}
\usepackage{mathrsfs}
\usepackage{tikz-cd}
\usepackage{import}
\usepackage[utf8]{inputenc}
\usepackage{hyperref}
\usepackage{cancel}
\usepackage[left=1in,top=1in,right=1in,bottom=1in]{geometry}
\usepackage{mathtools}

\usetikzlibrary{decorations.pathreplacing,calligraphy,decorations,decorations.markings}

\newtheorem{thm}{Theorem}[section]
\newtheorem{cor}[thm]{Corollary}
\newtheorem{prop}[thm]{Proposition}
\newtheorem{lem}[thm]{Lemma}
\newtheorem{conj}[thm]{Conjecture}

\theoremstyle{definition}

\newtheorem{exmp}[thm]{Example}

\newtheorem{rmk}[thm]{Remark}

\newcommand{\Out}{\operatorname{Out}}
\newcommand{\Aut}{\operatorname{Aut}}
\newcommand{\hAut}{\operatorname{hAut}}
\newcommand{\iso}{\operatorname{Iso}}
\newcommand{\sgn}{\mathrm{sgn}}

\newcommand{\br}{\mathrm{br}}
\newcommand{\Conf}{\operatorname{Conf}}
\newcommand{\into}{\hookrightarrow}
\newcommand{\coker}{\operatorname{coker}}
\newcommand{\End}{\operatorname{End}}

\DeclareMathOperator{\id}{id}
\newcommand{\gp}{\mathcal{G}}
\newcommand{\col}{\colon}
\newcommand{\Specht}{\chi}
\def\C{\mathbb{C}}

\def\Q{\mathbb{Q}}
\def\R{\mathbb{R}}
\def\Z{\mathbb{Z}}
\def\calM{\mathcal{M}}

\title[Homology representations of compactified configurations on graphs]{Homology representations of compactified configurations on graphs applied to $\mathcal{M}_{2,n}$}
\author{Christin Bibby}
\address{Department of Mathematics, Louisiana State University, Baton Rouge, LA 70803}
\email{\url{bibby@math.lsu.edu}}
\author{Melody Chan}
\address{Department of Mathematics, Brown University, Box 1917, Providence, RI 02912}
\email{\url{melody_chan@brown.edu}}
\author{Nir Gadish}
\address{Department of Mathematics, Massachusetts Institute of Technology, Cambridge, MA 02139}
\email{\url{ngadish@mit.edu}}
\author{Claudia He Yun}
\address{Department of Mathematics, Brown University, Box 1917, Providence, RI 02912}
\email{\url{he_yun@alumni.brown.edu}}



\begin{document}

\begin{abstract}
We obtain new calculations of the top weight rational cohomology of the moduli spaces $\mathcal{M}_{2,n}$, equivalently the rational homology of the tropical moduli spaces $\Delta_{2,n}$, as a representation of $S_n$.  These calculations are achieved fully for all $n\leq 11$, and partially---for specific irreducible representations of $S_n$---for $n\le 22$.   We also present conjectures, verified up to $n=22$, for the multiplicities of the irreducible representations $\mathrm{std}_n$ and $\mathrm{std}_n\otimes \mathrm{sgn}_n$. 

We achieve our calculations via a comparison with the homology of compactified configuration spaces of graphs. These homology groups are equipped with commuting actions of a symmetric group and the outer automorphism group of a free group.
In this paper, we construct an  
efficient
free resolution for these homology representations, from which 
we extract calculations on irreducible representations one at a time, simplifying the calculation of these homology representations. 
\end{abstract}
	
\maketitle

\section{Introduction}

\subsection{Main results}

The moduli spaces $\Delta_{g,n}$ of tropical curves are combinatorial moduli spaces which are canonically identified with the boundary complex of the Deligne-Mumford-Knudsen compactification $\overline{\calM}_{g,n}$ of the moduli spaces of algebraic curves. See \cite{ACP15} and \cite{CGP1}. Consequently, by work of Deligne (\cite{PMIHES_1971__40__5_0},\cite{PMIHES_1974__44__5_0}), there is a canonical $S_n$-equivariant isomorphism between $\widetilde H_*(\Delta_{g,n};\Q)$ and the top-weight rational cohomology of $\calM_{g,n}$: \begin{equation}
    \widetilde{H}_{k-1}(\Delta_{g,n};\Q) \cong \mathrm{Gr}_{6g-6+2n}^WH^{6g-6+2n-k}(\mathcal{M}_{g,n};\Q).
    \label{isom: Top weight cohomology of Mgn is identified with homology of Deltagn}
\end{equation}

In this work we compute, for genus $g=2$, the homology groups $\widetilde H_*(\Delta_{2,n};\Q)$ as representations of $S_n$ in a range beyond what was previously accessible, using an approach centered on a compactified graph configuration space.  

\begin{thm}\label{thm:main}
The rational homology $\widetilde{H}_*(\Delta_{2,n};\Q)$ is supported in degrees $* = n+1$ and $n+2$, with the character of $\widetilde{H}_{n+1}(\Delta_{2,n};\Q)$ as an $S_n$-representation for $n\leq 11$ given in Table \ref{table:data}.  
Partial irreducible decompositions of $\widetilde{H}_{n+1}(\Delta_{2,n};\Q)$ for $12\le n \le 17$ are given in Table~\ref{table:partialdata}.
\end{thm}

\noindent 
Given that the equivariant Euler characteristic of $\Delta_{2,n}$ is known (see \cite{CFGP}), Table \ref{table:data} is sufficient to determine the entire homology representation.  See \S\ref{sec:related work} for a discussion of previous related work relating graph complexes and compactified configuration spaces.

The first $8$ rows of Table \ref{table:data} were recently computed by the fourth author, see \cite{Yun2021}. Our current approach gives data well beyond what was feasible with those techniques. For example, even the dimension of $\widetilde{H}_{12}(\Delta_{2,11};\Q)$ was not known: it is 850732.

\begin{table}
{\small
\begin{tabular}{|r|p{.9\textwidth}|}
\hline
        $n$ & Character of $\widetilde{H}_{n+1}(\Delta_{2,n};\Q)$ \\
\hline  
0 & 0 \\
\hline
1 & 0 \\
\hline  
        2 & 0 \\
\hline        
        3 & 0 \\
\hline        
        4 & $\Specht_{(4)}$ \\
\hline        
        5 & $\Specht_{(3,2)}$ \\
\hline        
        6 & $\Specht_{(4,1^2)} + \Specht_{(3,2,1)}$ \\
\hline        
        7 & $\Specht_{(5,1^2)} + \Specht_{(4,3)} + \Specht_{(4,2,1)} + \Specht_{(4,1^3)} + \Specht_{(3^2,1)} + \Specht_{(3,2,1^2)} + \Specht_{(2^3,1)} + \Specht_{(1^7)}$ \\
\hline        
        8 & $\Specht_{(8)} + \Specht_{(6,2)} + \Specht_{(5,3)} + 2\Specht_{(5,2,1)} + \Specht_{(5,1^3)} + 2\Specht_{(4,3,1)} + 2\Specht_{(4,2^2)} + 2\Specht_{(4,2,1^2)} + \Specht_{(4,1^4)} + \Specht_{(3^2,1)} + \Specht_{(3^2,1^2)} + 2\Specht_{(3,2^2,1)} + 2\Specht_{(3,2,1^3)} + \Specht_{(3,1^5)}$ \\
\hline        
        9 & $2\Specht_{(7, 2)} + \Specht_{(6, 3)} + 3\Specht_{(6, 2, 1)} + \Specht_{(6,1^3)} + 2\Specht_{(5, 4)} + 3\Specht_{(5, 3, 1)} + 5\Specht_{(5, 2^2)} + 4\Specht_{(5, 2, 1^2)} + 3\Specht_{(5,1^4)} + 3\Specht_{(4^2, 1)} + 4\Specht_{(4, 3, 2)} + 5\Specht_{(4, 3, 1^2)} + 5\Specht_{(4, 2^2, 1)} + 4\Specht_{(4, 2, 1^3)} + \Specht_{(4, 1^5)} + 4\Specht_{(3^2, 2, 1)} + 4\Specht_{(3^2, 1^3)} + 3\Specht_{(3, 2^3)} + 2\Specht_{(3, 2^2, 1^2)} + 3\Specht_{(3, 2, 1^4)} + \Specht_{(2^4, 1)} + \Specht_{(2^3, 1^3)} + \Specht_{(2^2, 1^5)} + \Specht_{(1^9)}$ \\
\hline        
        10 & $2\Specht_{(8, 1^2)} + 2\Specht_{(7, 3)} + 4\Specht_{(7, 2, 1)} + 3\Specht_{(7,1^3)} + 2\Specht_{(6, 4)} + 9\Specht_{(6, 3, 1)} + 4\Specht_{(6, 2^2)} + 8\Specht_{(6, 2, 1^2)} + 2\Specht_{(6, 1^4)} + 7\Specht_{(5, 4, 1)} + 10\Specht_{(5, 3, 2)} + 15\Specht_{(5, 3, 1^2)} + 12\Specht_{(5, 2^2, 1)} + 9\Specht_{(5, 2, 1^3)} + 2\Specht_{(5, 1^5)} + 6\Specht_{(4^2, 2)} + 6\Specht_{(4^2, 1^2)} + 6\Specht_{(4, 3^2)} + 16\Specht_{(4, 3, 2, 1)} + 11\Specht_{(4, 3, 1^3)} + 7\Specht_{(4, 2^3)} + 13\Specht_{(4, 2^2, 1^2)} + 8\Specht_{(4, 2, 1^4)} + 3\Specht_{(4, 1^6)} + 6\Specht_{(3^3, 1)} + 4\Specht_{(3^2, 2^2)} + 10\Specht_{(3^2, 2, 1^2)} + 3\Specht_{(3^2, 1^4)} + 6\Specht_{(3, 2^3, 1)} + 7\Specht_{(3, 2^2, 1^3)} + 3\Specht_{(3, 2, 1^5)} + 2\Specht_{(3, 1^7)} + \Specht_{(2^4, 1^2)} + 2\Specht_{(2^3, 1^4)}$ \\
\hline 
        11 & $3\Specht_{(9, 1^2)}+
3\Specht_{(8, 3)}+
5\Specht_{(8, 2, 1)}+
3\Specht_{(8, 1^3)}+
2\Specht_{(7, 4)}+
16\Specht_{(7, 3, 1)}+
5\Specht_{(7, 2^2)}+
16\Specht_{(7, 2, 1^2)}+
2\Specht_{(7, 1^4)}+
4\Specht_{(6, 5)}+
15\Specht_{(6, 4, 1)}+
23\Specht_{(6, 3, 2)}+
28\Specht_{(6, 3,1^2)}+
24\Specht_{(6, 2^2, 1)}+
21\Specht_{(6, 2, 1^3)}+
5\Specht_{(6, 1^5)}+
10\Specht_{(5^2, 1)}+
19\Specht_{(5, 4, 2)}+
28\Specht_{(5, 4, 1^2)}+
21\Specht_{(5, 3^2)}+
50\Specht_{(5, 3, 2, 1)}+
28\Specht_{(5, 3, 1^3)}+
13\Specht_{(5, 2^3)}+
38\Specht_{(5, 2^2, 1^2)}+
17\Specht_{(5, 2, 1^4)}+
7\Specht_{(5, 1^6)}+
8\Specht_{(4^2, 3)}+
29\Specht_{(4^2, 2, 1)}+
20\Specht_{(4, 4, 1^3)}+
25\Specht_{(4, 3^2, 1)}+
28\Specht_{(4, 3, 2^2)}+
48\Specht_{(4, 3, 2, 1^2)}+
22\Specht_{(4, 3, 1^4)}+
22\Specht_{(4, 2^3, 1)}+
25\Specht_{(4, 2^2, 1^3)}+
11\Specht_{(4, 2, 1^5)}+
2\Specht_{(4, 1^7)}+
13\Specht_{(3^3, 2)}+
8\Specht_{(3^3, 1^2)}+
22\Specht_{(3^2, 2^2, 1)}+
20\Specht_{(3^2, 2, 1^3)}+
11\Specht_{(3^2, 1^5)}+
4\Specht_{(3, 2^4)}+
15\Specht_{(3, 2^3, 1^2)}+
8\Specht_{(3, 2^2, 1^4)}+
6\Specht_{(3, 2, 1^6)}+
3\Specht_{(2^5, 1)}+
4\Specht_{(2^4, 1^3)}+
2\Specht_{(2^3, 1^5)}+
2\Specht_{(2^2, 1^7)}+
\Specht_{(1^{11})}$ \\
\hline
\end{tabular}
}
\vspace{.2cm}

    \caption{Character of $\widetilde H_{n+1}(\Delta_{2,n};\Q)$ for $n\leq 11$.}
   \label{table:data}
\end{table}

\noindent 
Table \ref{table:partialdata} in \S\ref{sec:tabulation} shows the partial calculations for multiplicities of certain small $S_n$-irreducibles in the range $ 12 \le n \le 17$.  For $18 \leq n \leq 22$, we obtained multiplicities for $\Specht_{(n)}$, $\Specht_{(1^n)}$, $\Specht_{(n-1,1)}$ and $\Specht_{(2,1^{(n-2)})}$, and for $23 \le n \le 25$, we obtained multiplicities for $\Specht_{(n)}$, $\Specht_{(1^n)}$ only.  
The data is extensive enough to suggest patterns in the multiplicities of the standard representation $\chi_{(n-1,1)}$ and its sign twist $\chi_{(2,1^{n-2})}$. See Conjecture \ref{conj:std} and surrounding discussion.

We now outline the key steps to our calculations. 
Together, they establish Theorem \ref{thm:main}, the main theorem of this paper. 

\subsubsection{Reduction to compactified configurations on a theta graph}
We immediately leave the tropical world and work instead with $\Conf_n(G)^+$, the one-point compactification of the configuration space of $n$ distinct marked points on a graph $G$. 
In genus $g=2$, the tropical moduli space $\Delta_{2,n}$ is directly related to a single such compactified configuration space. 
Specifically, Theorem \ref{thm: red cohom of Delta2n and red cohom of Conf} establishes a homotopy equivalence inducing the following isomorphism of $S_n$-representations:
\begin{equation} \label{eq:intro-homology of Delta is Conf}
    \widetilde{H}_i(\Delta_{2,n};\Q) \cong (\sgn_3 \otimes \widetilde{H}_{i-2}(\Conf_n(\Theta)^+;\Q))_{\mathrm{Iso}(\Theta)},
\end{equation}
where $\Theta$ is the graph  with two vertices and three parallel edges between them, $\sgn_3$ is the sign representation of $S_3$ in the automorphism group $\mathrm{Iso}(\Theta) \cong S_2\times S_3$,
and the subscript $\iso(\Theta)$ denotes the coinvariant quotient.

\subsubsection{Reduction to compactified configurations on a rose graph}
Note that the graph $\Theta$ is homotopy equivalent to a wedge of two circles. More generally, any finite graph $G$ with first Betti number $g$ is homotopy equivalent to a rose graph $R_g = \vee_{g}S^1$.  In fact, 
a homotopy equivalence of compact Hausdorff spaces induces a homotopy equivalence of their compactified configuration spaces; the analogous statement is not true for uncompactified configuration spaces.
See Proposition \ref{prop: compactified conf_n is homotopy invariant}.  
So it suffices to work with $\Conf_n(R_2)^+$, or more generally $\Conf_n(R_g)^+$ for any $g$. 
Proposition \ref{prop: compactified conf_n is homotopy invariant} endows the homology $H_*(\Conf_n(R_g)^+;\Q)$ with a canonical action of the group $\Out(F_g)$ of outer automorphisms of the free group on $g$ letters.
Moreover, a consequence of Proposition \ref{prop:iso(G) acting through Out} will be that a homotopy equivalence $G\xrightarrow{\ \sim\ } R_g$ induces a group homomorphism $\iso(G)\to\Out(F_g)$ so that the induced isomorphism 
\[H_*(\Conf_n(G)^+;\Q)\cong H_*(\Conf_n(R_g)^+;\Q)\]
is $\iso(G)$-equivariant. 
This, along with \eqref{eq:intro-homology of Delta is Conf}, reduces the computation of $\widetilde{H}_*(\Delta_{2,n};\Q)$ 
to computing the actions of $S_n$ and $\Out(F_2)$ on the homology of $\Conf_n(R_2)^+$.

\subsubsection{Cellular decomposition of compactified configurations on a rose graph}
The remaining goal in \S\ref{sec:homology of conf} is then to understand  $\widetilde{H}_*(\Conf_n(R_g)^+)$ as a representation of both $S_n$ and $\Out(F_g)$. The fundamental tool is an $S_n$-equivariant cell structure on the configuration space $\Conf_n(R_g)^+$, in which cells are permuted freely. This structure implies that the homology of $\Conf_n(R_g)^+$ is computed by a 2-step complex of free $\Z[S_n]$-modules, where the boundary map and the action of $\Out(F_g)$ are represented by explicit matrices with entries in $\Z[S_n]$. See Lemmas \ref{lem:resolution}, \ref{lem:boundary}, and \ref{lem:out acting on chains}.

\subsubsection{Improved computational efficiency through representation theory}
The presentation of homology by free $S_n$-modules allows for particularly efficient computations. Indeed, specializing to rational coefficients, Schur's lemma lets us work one irreducible at a time, performing any homology calculation at the level of multiplicity spaces of individual irreducible representations of $S_n$. See Lemma \ref{lem:splitting irreducibles}. 
This reduces 
the size of the matrices involved by a factor of at least $\sqrt{n!}$.

\medskip

Finally, after the above reductions, we implemented the resulting calculation  in SageMath, from which we obtained the data in Tables \ref{table:data} and \ref{table:partialdata} that prove Theorem \ref{thm:main}.  See \S\ref{sec:tabulation} for more details on the SageMath computations.

\subsection{Related work}\label{sec:related work}
Our initial motivation in this paper comes from tropical geometry, particularly the connection to cohomology of moduli spaces of curves.  Our calculations, however, 
are also connected to several other topics in geometry and topology, adding potential interest to our work.  We touch on several of them here: spaces of long embeddings and string links; modular operads; and representations of mapping class groups. We remark that our techniques do not apply to the {\em uncompactified} configuration spaces of graphs.

\subsubsection{$S_n$-equivariant homology of $\Delta_{g,n}$}
Here is a brief survey of previous calculations. The case $g=2$ is the first case in which the topology of $\Delta_{g,n}$ is not fully understood. 
\begin{itemize}
    \item When $g=0$ and $n\geq 4$, \cite{ROBINSON1996245} prove that $\Delta_{0,n}$ has homotopy type of a wedge of spheres of dimension $n-3$ and give a formula for the character of the $S_n$-representation occurring in the top degree integral homology $H_{n-3}(\Delta_{0,n};\Z)$.
    \item When $g=1$ and $n>0$,
       \cite{getzler1999} computes the $S_n$-equivariant Serre characteristic of $\mathcal{M}_{1,n}$, from which the character of $H_{n-1}(\Delta_{1,n};\Q)$ can also be derived. 
    Moreover, \cite{cgp-marked} prove that $\Delta_{1,n}$ has homotopy type of a wedge of spheres of dimension $n-1$. 
    \item     When $g=2$, 
    \cite{chan2015topology} proves that the homology of $\Delta_{2,n}$ is concentrated in its top two degrees, and computes numerically the Betti numbers for $n\leq 8$. 
    \cite{Yun2021} computes these homology groups $S_n$-equivariantly. 
    \item For all $g,n\ge 0$ with $2g-2+n>0$, \cite{CFGP} proves a general formula for the $S_n$-equivariant Euler characteristic for $\Delta_{g,n}$, as conjectured by D. Zagier.
\end{itemize}

\subsubsection{Spaces of long embeddings and string links}

The rational homotopy type of spaces of ``long embeddings" $\operatorname{Emb}_c(\R^m,\R^n)$ is given by the homology of certain ``hairy graph complexes''  introduced by Arone-Turchin \cite{arone-turchin-graph-complexes}. These complexes have a geometric interpretation as homology with local coefficients of the tropical moduli spaces, as we will explain more in forthcoming work.    
These complexes in fact depend only on the parity of $m$ and $n$, up to degree shift. When $n$ is even and $m$ is odd, the decoration attached to each graph is the Hochschild-Pirashvili homology of the graph, which is equivalent to the collection of $\Out(F_g)$-representations on the $S_k$-invariant parts of $\widetilde{H}_*(\Conf_k(R_g)^+)$ for all $k$, which we study below in \S\ref{sec:homology of conf} (see \cite[Theorem 1]{gadish-hainaut}).
Similarly, our sign multiplicity spaces coincide with the hairy graph homology when $n$ and $m$ are both even.
See \cite[Remark 5.2]{turchin2017commutative} for applications of the other isotypic components to rational homotopy groups of the space of string links, and
see \cite[Section 2.5]{turchin-willwacher-hochschild} for an interpretation of the isotypic components as the {\em bead representations}.

The reason the above two complexes ($n$ even, $m$ even or odd) only relate to our trivial and sign computations is that the ``hairs'' in hairy graph complexes are unlabeled.  In \cite{tsopmene-turchin-euler}, the authors study spaces of string links via complexes of graphs with labeled hairs (possibly with labels repeated or missing). These are equivalent to ours in the sense that ours are a special case, while theirs can be obtained from ours by taking invariants under Young subgroups of symmetric groups.

In genus $2$ specifically, we refer to the work of Conant--Costello--Turchin--Weed \cite{Conant-Costello-Turchin-Weed}, who show that only the graph $\Theta$ contributes to the hairy graph homology, which furthermore takes the form $(\sgn_3\otimes V)_{S_2\times S_3}$ for some $V$ computed by a $3$-step complex. This echos our Theorem \ref{thm: red cohom of Delta2n and red cohom of Conf}, and in fact 
proves the specialization  
to the trivial and sign isotypic components. Our further reduction in this paper from $\Theta$ to $R_2$ (and indeed from any graph to $R_g$), as well as the richer structure coming from the $S_n$-action on the $n$ labels, is not studied in that paper. See also Remark \ref{rmk:triv-sgn} for further connections.

\subsubsection{Representations of mapping class groups}
In a different direction, Moriyama \cite{moriyama} studies representations of the mapping class group of a surface of genus $g$ with one boundary component. 
These representations are the cohomology of the compactified configuration space on the surface with an additional point removed from the boundary. Since a punctured surface is not compact, the compactified configuration space is not homotopy equivalent to $\Conf_n(R_{2g})^+$ (in contrast with Proposition \ref{prop: compactified conf_n is homotopy invariant}), and has homology concentrated in degree $n$ only. 
Nevertheless, Moriyama \cite[Section 4]{moriyama} accesses his cohomology using a cell structure whose only nontrivial cells are $n$-cells, which are exactly the top-dimensional cells that we consider below. In particular, his setup does not include $(n-1)$-cells, the existence of which constitutes the central computational challenge in our work.

\subsubsection{Modular operads}
We remark briefly on the relationship to modular operads \cite{getzler-kapranov-modular}, postponing details to a sequel.  The cellular chain complex of the moduli space $\Delta_{g,n}$ is isomorphic to the Feynman transform $F\mathsf{ModCom}((g,n))$, where $\mathsf{ModCom}$ is the modular-commutative operad $\mathsf{ModCom}((g,n))=\Q$ in degree 0 for each $(g,n)$ with $2g-2+n>0$. In fact $F\mathsf{ModCom}((g,n))$ is quasi-isomorphic to $F\mathsf{Com}((g,n))$ whenever $g>0$ and $(g,n)\ne (1,1)$; see \cite[Remark 3.3]{cgp-marked}.  Here, $\mathsf{Com}$ is the commutative operad $\mathsf{Com}((g,n)) = \Q$ in degree $0$ for each $g=0$ and $n\ge 3$,  and $0$ otherwise.  In light of our Theorem \ref{thm: red cohom of Delta2n and red cohom of Conf}, our results give computations of the homology of $F\mathsf{ModCom}((2,n))$ and of $F\mathsf{Com}((2,n))$ in the range $n\le 22$. These have renewed interest in light of the recent results of \cite{CGP1, cgp-marked}.

\subsubsection{Future work}

A sequel to this paper shall present computations on genus $g>2$ graph complexes in relation to $\Delta_{g,n}$, via a Serre-like spectral sequence whose $E_1$ page involves the compactified configuration spaces of more than one graph of genus $g$.  In that paper we will also treat more precisely the connections between modular operads, cellular chains of $\Delta_{g,n}$, and hairy graph complexes that are sketched above.  It would be interesting to extend the computations in this paper to explore the other parities (of $n$ and $m$) of graph complexes.

\subsection*{Acknowledgements} We thank Eric Ramos for informing us of each others' work.  
We also thank Dan Petersen and Louis Hainaut for suggesting to us the connection between our configuration spaces and Hochschild homology, along with many other useful ideas.
We thank Victor Turchin, Ronno Das, Philip Tosteson, Orsola Tommasi, and Ben Ward for helpful conversations. Lastly, we thank ICERM and Brown University for generously providing us with the computing resources on which we ran our program.
C.B. was supported by NSF DMS-2204299; N.G. was supported by NSF Grant No. DMS-1902762; M.C. was supported by NSF DMS-1701924, CAREER DMS-1844768, and a Sloan Research Fellowship.

\section{Homology of compactified configuration spaces of graphs} \label{sec:homology of conf}

For a topological space $X$ and for $n\ge 0$, recall the configuration space \[\Conf_n(X) = \{ (x_1,\ldots,x_n)\in X^n \mid x_i\ne x_j \textrm{ for all } i\ne j\}.\] 
We refer to the one-point compactification, denoted $\Conf_n(X)^+$,   as the {\em compactified configuration space}. 
In this paper, we only consider compactified configuration spaces on compact Hausdorff spaces $X$. In this case, there is an $S_n$-equivariant homeomorphism of pointed spaces
\begin{equation}
    \Conf_n(X)^+ \cong {X^n}/{ \{ (x_1,\ldots,x_n)\in X^n \mid x_i=x_j \textrm{ for some } i\ne j\}}.
\end{equation}

\begin{prop} 
For each $n\ge 0$, $\Conf_n(-)^+$ is a functor from the category of compact Hausdorff topological spaces and \emph{all} continuous maps  
to the category of pointed topological spaces with $S_n$-action.  
Moreover, if $f,g\col X\to Y$ are  
homotopic, then the induced maps \[\Conf_n(X)^+\to \Conf_n(Y)^+\] are again homotopic.
\label{prop: compactified conf_n is homotopy invariant}
\end{prop}
\noindent Thus, in contrast to the situation for uncompactified configuration spaces, for a compact $X$ (such as a finite graph), the homotopy type of $\Conf_n(X)^+$ depends only on the homotopy type of $X$, and a self homotopy equivalence of $X$ induces an $S_n$-equivariant self homotopy equivalence of  $\Conf_n(X)^+$.
In fact, a more general version of Proposition~\ref{prop: compactified conf_n is homotopy invariant} is true: for (not necessarily compact) Hausdorff spaces, the functor $\Conf_n(-)^+$ takes a \emph{proper} homotopy equivalence to a proper homotopy equivalence.
\begin{proof}[Proof of Proposition~\ref{prop: compactified conf_n is homotopy invariant}]
Let $X$ and $Y$ be compact Hausdorff, and let $f\col X\to Y$ be a continuous map.  
Write $f^n\col X^n\to Y^n$ for the induced map of Cartesian powers.
The source and target spaces contain copies of $\Conf_n(X)$ and $\Conf_n(Y)$, respectively, and the preimage of $\Conf_n(Y)$ is contained in $\Conf_n(X)$. Therefore, collapsing the complements of $\Conf_n(X)$ and $\Conf_n(Y)$  yields the desired pointed $S_n$-equivariant map 
 $$\Conf_n(X)^+ \to \Conf_n(Y)^+.$$  
 
 Moreover, if $F\col X\times [0,1]\to Y$ is a homotopy  
 between $f$ and $g$, then in the same manner we obtain an $S_n$-equivariant map $X^n\times [0,1]\to Y^n$, and an $S_n$-equivariant homotopy 
 \[\Conf_n(X)^+\times [0,1]\to \Conf_n(Y)^+\]
 between the maps induced by $f$ and $g$.
\end{proof}

Homotopy invariance of $\Conf_n(-)^+$ in particular gives well-defined and natural actions of the groups of homotopy automorphisms, as the following proposition explains.

\begin{prop}\label{prop:iso(G) acting through Out}Let $X$ and $Y$ be compact Hausdorff, and let 
$\hAut(X)$ and $\hAut(Y)$
be their respective groups of homotopy classes of self-homotopy equivalences. Let $X\xrightarrow{\ \sim\ }Y$ be a homotopy equivalence, and let 
$\phi\col \hAut(X) \to \hAut(Y)$
be the induced group homomorphism.  Then the induced isomorphism of graded $S_n$-representations
\[
H_*(\Conf_n(X)^+)\xrightarrow{\ \sim\ } H_*(\Conf_n(Y)^+)
\]
is $\hAut(X)$-equivariant, where $\hAut(X)$ acts on the right-hand side through $\phi$.

\end{prop}
\begin{proof}

Let $m\colon X\xrightarrow{\ \sim\ }Y$ and $m'\colon Y\xrightarrow{\ \sim\ } X$ be inverse homotopy equivalences. Then any $f\in \hAut(X)$ determines an auto-equivalence $mfm': Y\xrightarrow{\ \sim\ } Y$, and therefore an element of $\hAut(Y)$, and this association descends to a well-defined map $\hAut(X)\to \hAut(Y)$.
Given another $g\in \hAut(X)$, the composition $(mfm')(mgm') = mf(m'm)gm'$ is homotopic to $mfgm'$, yielding that $\phi$ is a homomorphism.
Functoriality and homotopy invariance of $\Conf_n(-)^+$ gives the compatibility of the two actions.
\end{proof}

Specializing to finite graphs, the above facts show that a calculation of $H_*(\Conf_n(G)^+)$ for just one graph $G$ along with the induced action of $\hAut(G)$ determines the analogous representations for all other homotopy equivalent graphs. One may then work with the simplest graph of a given genus, as we do next.

\subsection{An $S_n$-equivariant cell structure on $\Conf_n(R_g)^+$}

Let $R_g=\vee_{i=1}^g S^1$ be the ``rose graph:'' a wedge of $g$ circles, with the unique vertex denoted $v$.  
For any finite, connected graph  $G$ of genus (first Betti number) $g=|E(G)|-|V(G)|+1$,
we may use Proposition \ref{prop: compactified conf_n is homotopy invariant} above to compute $\widetilde{H}_*(\Conf_n(G)^+)$ with its natural action of $\iso(G)$, via computing $\widetilde{H}_*(\Conf_n(R_g)^+)$ with its natural $\hAut(R_g)$-action.

We now fix a cellular structure on $\Conf_n(R_g)^+$, from which we obtain a 2-step free resolution for $\widetilde{H}_*(\Conf_n(R_g)^+;\Z)$ as an (integral) $S_n$-representation used in this paper.\footnote{We've learned though private communication that O. Tommasi, D. Petersen and P. Tosteson have independently found the same construction for this calculation.  Petersen and Tommasi have also obtained results on the weight-0 compactly supported cohomology of $\mathcal{M}_{2,n}$, also using graph calculations.  At this moment, we do not know how to directly relate their methods with the ones presented in this paper.}

Let $\Xi_g = \bigcup_{i=1}^g (i-1,i)\subset\R$ be a union of $g$ open intervals, and fix a  homeomorphism to $R_g\setminus\{v\}$. We sometimes call the intervals arcs since they correspond to the arcs of the petals in $R_g$ after removing the central vertex. Denote $[n]=\{1,2,\dots,n\}$, and for $S\subseteq[n]$ let $\Conf_S(\Xi_g)$ be the space of configurations of points in $\Xi_g$ with labels in $S$. 
Then $\Conf_S(\Xi_g)$ decomposes as a disjoint union of open polyhedra as follows. Let $|S|=k$; then for every pair $(\sigma, \chi)$, where 
$\sigma: [k]\xrightarrow{\ \cong\ } S$ is a total ordering on $S$ and $\chi\colon [k]\to [g]$ a nondecreasing function, we associate the collection of configurations 
$(x_s)_{s\in S}\in \Conf_S(\Xi_g)$, 
where
\[
x_{\sigma_a} < x_{\sigma_b} \in \R \iff a < b \in [k] \quad \text{ and }  \quad x_{\sigma_a}\in (i-1,i) \iff \chi(a) = i.
\]
Writing $\sigma_i :=\sigma(i)$ for short, we denote this collection of configurations by
\begin{equation}\label{eq:cellnames}
(\sigma_1\sigma_2\ldots\sigma_{j_1}|\sigma_{j_1+1}\ldots \sigma_{j_2}| \ldots | \ldots | \ldots  \sigma_{j_{g-1}}| \sigma_{j_{g-1}+1}\ldots \sigma_k).
\end{equation}
where $\chi^{-1}(1) = \{1,\ldots,j_1\},$ $\chi^{-1}(2) = \{j_1+1,\ldots,j_2\},$
and so on. Set $j_0 = 1$ and $j_g = k$.

\begin{exmp}\label{ex:|123|4|56|}
For $n=6$ and $g=3$, $(413|5|62)$ denotes the collection of configurations of points $(x_1,x_2,\dots,x_6)\in\R^6$ with $0<x_4<x_1<x_3<1<x_5<2<x_6<x_2<3$. One such configuration can be pictured as:
\begin{center}
\begin{tikzpicture}[scale=1.3]
\foreach \x in {0,1,2,3} {
\node[draw,circle,fill=none,minimum size=4pt,inner sep=0pt] (\x) at (2*\x,0) {};
};
\draw[thick,-] (0)--(1);
\draw[thick,-] (1)--(2);
\draw[thick,-] (2)--(3);
\foreach \x in {0.5,1,1.5,3,4.7,5.3} {
\node[draw,circle,fill=black!50,minimum size=4pt,inner sep=0pt]
at (\x,0) {};
};
\node at (0.5,.2) {\scriptsize 4};
\node at (1,.2) {\scriptsize 1};
\node at (1.5,.2) {\scriptsize 3};
\node at (3,.2) {\scriptsize 5};
\node at (4.7,.2) {\scriptsize 6};
\node at (5.3,.2) {\scriptsize 2};
\end{tikzpicture}
\end{center}

Intervals may be vacant, as in the case of $(321||654)$, which contains configurations without points on the second interval, such as:
\begin{center}
\begin{tikzpicture}[scale=1.3]
\foreach \x in {0,1,2,3} {
\node[draw,circle,fill=none,minimum size=4pt,inner sep=0pt] (\x) at (2*\x,0) {};
};
\draw[thick,-] (0)--(1);
\draw[thick,-] (1)--(2);
\draw[thick,-] (2)--(3);
\foreach \x in {0.5,1,1.5,4.5,5,5.5} {
\node[draw,circle,fill=black!50,minimum size=4pt,inner sep=0pt]
at (\x,0) {};
};
\node at (0.5,.2) {\scriptsize 3};
\node at (1,.2) {\scriptsize 2};
\node at (1.5,.2) {\scriptsize 1};
\node at (4.5,.2) {\scriptsize 6};
\node at (5,.2) {\scriptsize 5};
\node at (5.5,.2) {\scriptsize 4};
\end{tikzpicture}
\end{center}
\end{exmp}

The configurations corresponding to  each $(\sigma,\chi)$ are parametrized by the interior of a 
product of open simplices.   As the following lemma states, this determines a cellular decomposition of $\Conf_n(R_g)^+$. Figure \ref{fig:n=2} illustrates this decomposition in the case $n=g=2$, omitting the point $\infty$.

\begin{figure}[ht]
\begin{tikzpicture}[scale=2]
\draw[dashed,-] (0.1,0.05) -- (0.95,.05) -- (0.95,0.9)-- (0.1,0.05);
\draw[dashed,-] (0.05,0.1) -- (0.05,0.95) -- (0.9,0.95) -- (0.05,0.1);
\draw[dashed,-] (1.1,1.05) -- (1.95,1.05) -- (1.95,1.9)-- (1.1,1.05);
\draw[dashed,-] (1.05,1.1) -- (1.05,1.95) -- (1.9,1.95) -- (1.05,1.1);
\draw[dashed,-] (1.05,0.05)--(1.05,0.95)--(1.95,0.95)--(1.95,0.05)--(1.05,0.05);
\draw[dashed,-] (0.05,1.05)--(0.05,1.95)--(0.95,1.95)--(0.95,1.05)--(0.05,1.05);
\foreach \x in {0,1,2} {
\draw[thick,-,blue] (\x,0.05)--(\x,0.95);
\draw[thick,-,OliveGreen] (\x,1.05)--(\x,1.95);
\draw[thick,-,red] (0.05,\x)--(0.95,\x);
\draw[thick,-,Purple] (1.05,\x)--(1.95,\x);
};
\node at (1.6,0.6) {\scriptsize $(2|1)$};
\node at (0.4,1.6) {\scriptsize $(1|2)$};
\node at (0.3,0.7) {\scriptsize $(12|)$};
\node at (0.7,0.3) {\scriptsize $(21|)$};
\node at (1.3,1.7) {\scriptsize $(|12)$};
\node at (1.7,1.3) {\scriptsize $(|21)$};
\node[blue] (a) at (0.7,-.3) {\scriptsize $(2|)$};
\node[red] (b) at (-.3,0.7) {\scriptsize $(1|)$};
\node[OliveGreen] (c) at (0.7,2.3) {\scriptsize $(|2)$};
\node[Purple] (d) at (2.3,0.7) {\scriptsize $(|1)$};
\draw (a.east) edge[bend right,->] (1,-0.1);
\draw (b.north) edge[bend left,->] (-0.1,1);
\draw (c.east) edge[bend left,->] (1,2.1);
\draw (d.north) edge[bend right,->] (2.1,1);
\end{tikzpicture}
\caption{Cellular decomposition of $\Conf_2(R_2)^+$, omitting $\infty$. 
The symmetric group $S_2$ acts on this picture via reflection across the diagonal.}
\label{fig:n=2}
\end{figure}
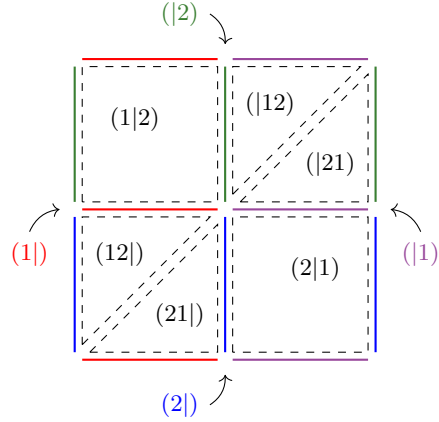

\begin{lem} \label{lem:resolution}
The space $\Conf_n(R_g)^+$ admits a cellular decomposition with a single $0$-cell,
$n!\cdot\binom{n+g-1}{g-1}$
cells in dimensions $n$, and 
$n!\cdot\binom{n+g-2}{n-1}$
cells in dimension $n-1$.  For $k\in \{n-1, n\}$, the $k$-dimensional cells are labelled by total orderings of $\{1,\ldots,k\}$, separated by $g-1$ bars, as denoted in \eqref{eq:cellnames}.  The natural $S_n$-action freely permutes the $(n-1)$- and $n$-cells.

Consequently, we have a chain complex of free $\Z[S_n]$-modules
\begin{equation}\label{eq:2-step resolution}
    \Z[S_n]^{\binom{n+g-1}{g-1}} \overset{\partial}{\to} \Z[S_n]^{\binom{n+g-2}{g-1}},
\end{equation}
where the modules are placed in degrees $n$ and $n-1$ respectively,  
whose homology is isomorphic, $S_n$-equivariantly, to the reduced homology $\widetilde{H}_*(\Conf_n(R_g)^+)$.
\end{lem}

\begin{proof}
Let $X^\bullet \subset \Conf_n(R_g)$ denote the closed subset of all configurations in which the vertex $v$ is inhabited, and let $X^{\circ} \subset \Conf_n(R_g)$ be its complement, parametrizing all configurations in which $v$ is uninhabited.
A choice of  homeomorphism $\Xi_g \cong R_g\setminus\{v\}$ yields an $S_n$-equivariant homeomorphism
\[
X^\circ \cong \Conf_n(\Xi_g).\]
Similarly, we obtain an $S_n$- 
equivariant homeomorphism
\[
 X^\bullet \cong \coprod_{|S|=n-1}\Conf_S(\Xi_g),
\]
where a configuration $(x_1,\ldots,x_n)\in X^\bullet$ in which $x_i=v$ determines a configuration in  
$\Conf_{[n]\setminus\{i\}}(\Xi_g)$,  and vice versa.

Following the discussion preceding the lemma statement, $X^\circ$ and $X^\bullet$ are disjoint unions of interiors of convex polyhedra in $\R^S$ for $|S|=n-1$ and $|S|=n$, each indexed by a pair $(\sigma,\chi)$. In this way we obtain the claimed cell structure on $\Conf_n(R_g)^+$. Now notice that $S_n$ acts freely on the $n$-cells and the $(n-1)$-cells, respectively.  
Therefore, the reduced cellular chain complex is quasi-isomorphic to the claimed $2$-step complex of free $S_n$-modules, and computes $\widetilde{H}_*(\Conf_n(R_g)^+;\Z)$ equivariantly with respect to $S_n$.
\end{proof}

As an immediate corollary, we have the following formula for the $S_n$-equivariant Euler characteristic of $\Conf_n(G)^+$ for any graph $G$.
\begin{cor}
Fix $g\geq 1$ and 
$n\geq 1$.  For any connected graph $G$ with first Betti number $g$, the $S_n$-equivariant Euler characteristic of $\Conf_n(G)^+$ in the representation ring of $S_n$ is
\begin{equation}
    (-1)^{n}\binom{n+g-2}{g-2}[\Z[S_n]].
\end{equation}
\end{cor}

\subsection{Explicit description of the 2-step complex}

In order to implement the 2-step complex that arises in~\eqref{eq:2-step resolution} in computer calculations,
we first explicitly orient the cells in the cellular decomposition of $\Conf_n(R_g)^+$. The open cells of $\Conf_n(\Xi_g) \subset \R^n$ are open subsets of $\R^n$, and inherit their orientation from the standard orientation of $\R^n$.  
For a set $S = [n]\setminus \{j\}$, first orient $\R^S$ so that 
the ordered basis $(e_1,\ldots,\widehat{e_j},\ldots,e_n)$  has sign $(-1)^{j-1}$.
Then orient the open cells of $\Conf_S(\Xi_g) \subset \R^S$ by restriction. This choice ensures that transpositions in $S_n$ always act by reversing orientation.

Then a permutation $\tau\in S_n$ sends the configuration $(x_s)_{s\in S}$ to the configuration $(x_{\tau^{-1}(t)})_{t\in \tau(S)}$. Label cells by pairs $(\sigma,\chi)$ as before; $\tau$ permutes cells according to \[(\sigma,\chi)\mapsto \sgn(\tau)(\tau^{-1}\circ \sigma, \chi),\] where the sign indicates orientation reversal. 

The set $\{(\id,\chi)\, | \, \chi:[k]\to [g] \text{ nondecreasing },k=n-1,n \}$ forms a set of representatives of $S_n$-orbits of cells. They give an equivariant isomorphism of the cellular chain complex $\Z[S_n]^{\binom{k+g-1}{g-1}} \xrightarrow{\ \sim\ } C_k^{\operatorname{CW}}$.
Explicitly, the action of $\sigma\in S_n$ on a representative is given by
\begin{equation*}
    \sigma\cdot (12\ldots j_1| \ldots | j_{g-1}+1 \ldots n) = \sgn(\sigma)(\sigma^{-1}_1 \sigma^{-1}_2 \ldots \sigma^{-1}_{j_1}| \ldots | \sigma^{-1}_{j_{g-1}+1}\ldots \sigma^{-1}_n),
\end{equation*}
hence gives rise to the identification between cells and permutations
\begin{equation} \label{eq:identifying cells with permutations}
    (\sigma_1\ldots\sigma_{j_1} | \sigma_{j_1+1}\ldots \sigma_{j_2}| \ldots | \ldots | \sigma_{j_{g-1}+1}\ldots \sigma_n) \longleftrightarrow \sgn(\sigma)\sigma^{-1} \in \Z[S_n]
\end{equation}
in the appropriate summand. 

Next, to describe the boundary operator explicitly, consider an open cell of $\Conf_n(\Xi_g)\into \Conf_n(R_g)$. As mentioned above, this is the interior of a polytope, and its boundary is a sum of open cells in $\coprod_{|S|=n-1}\Conf_S(\Xi_g)$.

\begin{lem} \label{lem:boundary}
The boundary operator on cells is given by 
\begin{equation}\label{eq:differential}
    \partial (\sigma_{1}\ldots |  \ldots | \ldots \sigma_{n}) = \sum_{i=1}^g (\ldots|\widehat{\sigma_{j_{i-1}+1}}\ldots \sigma_{j_i}|\ldots) -(\ldots|\sigma_{j_{i-1}+1}\ldots \widehat{\sigma_{j_i}}|\ldots).
\end{equation}
\end{lem}

\begin{proof}
The boundary operator on a top-dimensional cell indexed by $(\sigma_{1}\ldots |  \ldots | \ldots \sigma_{n})$ gives a signed sum of codimension $1$ cells that arise when one of the marked points on an edge of $R_g$ falls onto the vertex $v$.  All other collisions of points are identified with the $0$-cell $\infty$.  Thus~\eqref{eq:differential} follows, up to a verification of signs that we omit.
\end{proof}

\begin{exmp}
For $n=6$ and $g=3$, the cell $(123|4|56)$ 
has boundary given by 
\[
\partial (123|4|56) = (23|4|56) - (12|4|56) + \cancel{(123||56)} - \cancel{(123||56)} + (123|4|6) - (123|4|5).
\]
In particular, one observes that intervals that contain exactly one point do not contribute to the boundary. This is consistent with the observation that a point looping around a vacant edge in $R_g$ contributes no boundary.
\end{exmp}

\subsection{Action of homotopy equivalences $\Out(F_g)$} \label{subsec:Out action}

Let $\Out(F_g)$ denote the group of outer automorphisms of the free group on $g$ generators. Recall that 
$\Out(F_g) \cong \hAut(R_g).$
Therefore, by Proposition~\ref{prop: compactified conf_n is homotopy invariant}, there is an $\Out(F_g)$-action on the homology of $\Conf_n(R_g)^+$, which we describe here.

Fix generators $a_1,\dots,a_g$ for $F_g$. The group $\Out(F_g)$ is generated by the following automorphisms (see e.g. \cite{AFV2008}): 
{\em flips} $f_i$ for $i=1,\dots,g$; {\em swaps} $s_i$ for $i=1,\ldots,g-1$; and a {\em transvection} $t_{12}$, defined as follows: 
    \[f_i(a_j)=\begin{cases}
    a_i^{-1} & i=j\\
    a_j & i\neq j,
    \end{cases}\qquad \quad 
    s_i(a_j) = \begin{cases}
    a_{i+1} & j=i\\
    a_i & j=i+1\\
    a_j & j\neq i,i+1,
    \end{cases} \qquad
    t_{12}(a_j) = \begin{cases}
    a_1a_2 & j=1\\
    a_j & j\neq 1.
    \end{cases}
    \]

Note that $\Out(F_g)$ does not act on the space $R_g$, nor does it act on its cellular chains. Instead, the $\Out(F_g)$-action on homology is induced by a collection of continuous maps $R_g\to R_g$ that only satisfy the relations in $\Out(F_g)$ up to homotopy. Having picked generators $(\{f_i\},\{s_i\},t_{12})$, the $\Out(F_g)$-action is completely described by continuous realizations of these elements. In what follows, we denote such realizations and their operation on cellular chains by the corresponding uppercase letters $(\{F_i\},\{S_i\}, T_{12})$.

\begin{lem} \label{lem:out acting on chains} The actions of flips, swaps and transvections on homology can be realized by maps $R_g\to R_g$ that fix the vertex, and thus induces cellular maps on $\Conf_n(R_g)^+$. Their effect on cellular chains in the two nontrivial dimensions are given as follows.

The maps inducing flip and the swap permute the 
cells of $\Conf_n(R_g)^+$ as
\begin{align}
    &F_i:(\ldots|\textcolor{blue}{(j_{i-1}\!+\!1)} \textcolor{purple}{(j_{i-1}\!+\!2)}\ldots \textcolor{red}{j_i} | \ldots) \mapsto (-1)^{j_i-j_{i-1}}(\ldots|\textcolor{red}{j_i}\ldots\textcolor{purple}{(j_{i-1}\!+\!2)}\textcolor{blue}{(j_{i-1}\!+\!1)}| \ldots) \\
    & S_i: (\ldots|\textcolor{blue}{(j_{i-1}\!+\!1)\ldots, j_i} | \textcolor{red}{(j_{i}\!+\!1)\ldots j_{i+1} }| \ldots) \mapsto (\ldots|\textcolor{red}{(j_{i}\!+\!1)\ldots j_{i+1}} | \textcolor{blue}{(j_{i-1}\!+\!1)\ldots j_i }| \ldots).
\end{align}

The transvection $t_{12}$ is induced by the cellular operator
\begin{equation}\label{eq:transvection}
    T_{12}:(12\ldots j_1|\ldots j_2|j_2+1 \ldots | \ldots) \longmapsto \sum_{k=0}^{j_1} \sum_{\sigma\in \Psi_{k}}(12\ldots k| \sigma_{k+1}\ldots \sigma_{j_2}|j_2+1 \ldots| \ldots)
\end{equation}
where $\Psi_{k}$ is the set of shuffles of the ordered tuples $(k+1,\ldots, j_1)$ and $(j_1+1,\ldots ,j_2)$.
\end{lem}
\begin{proof}
The flip and swap are realized by simple linear maps on the intervals $(i-1,i)\subset \R$, hence reorder the points in the claimed manner.  
Note that 
the flip $F_i$ reverses the direction of the $i$-th arc,  
inducing an orientation shift of $(-1)^{j_i-j_{i-1}}$.

The transvection $t_{12}$ is realized by a map $T_{12}:R_g\to R_g$ that stretches the first arc to twice its original length, then lays the latter half along the second arc. Any points that inhabit this latter half get distributed along  
the second arc. 
The locus of configurations in which a point lands exactly on $1\in (0,2)$, or on an existing point in the configuration, belongs to a lower dimensional skeleton of $\Conf_n(R_g)^+$, and therefore do not contribute to calculations on cellular chains. 
Note also that the stretch is an orientation-preserving linear map. 
Hence all cells map to other cells with degree $0$ or $1$, and the ones in the image have points $1,\ldots, k$
 on the first arc, for some $k\le j_1$, and
 some shuffle of the points $k+1,\ldots,j_1$ and $j_1+1,\ldots,j_2$ on the second arc.
\end{proof}

\begin{exmp}\label{ex:transvection}
Recall the case $n=2$ and $g=2$ 
depicted in  Figure \ref{fig:n=2}.
Figure~\ref{fig:transvection} depicts 
the transvection operation $T_{12}$ on the cells $(1|2)$ and $(12|)$, respectively, where stretching the first arc by a factor of 2 consequently stretches the cells so that they cover the cells appearing in the formula \eqref{eq:transvection}.

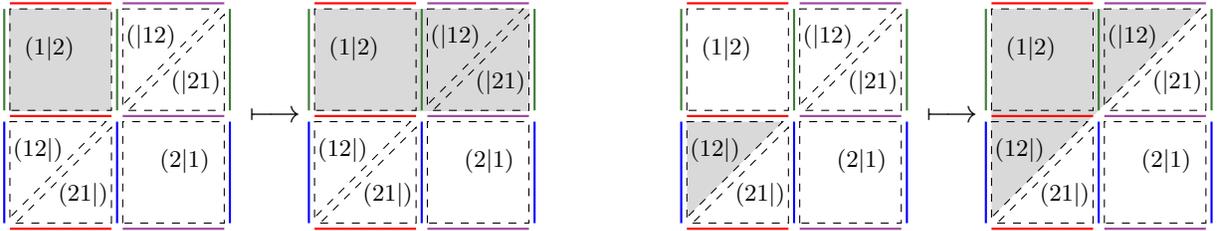
\begin{figure}[ht]
\begin{tikzpicture}
\begin{scope}[scale=1.5]
\draw[draw=none,fill=gray!30] (.05,1.05) rectangle (.95,1.95);
\draw[dashed,-] (0.1,0.05) -- (0.95,.05) -- (0.95,0.9)-- (0.1,0.05);
\draw[dashed,-] (0.05,0.1) -- (0.05,0.95) -- (0.9,0.95) -- (0.05,0.1);
\draw[dashed,-] (1.1,1.05) -- (1.95,1.05) -- (1.95,1.9)-- (1.1,1.05);
\draw[dashed,-] (1.05,1.1) -- (1.05,1.95) -- (1.9,1.95) -- (1.05,1.1);
\draw[dashed,-] (1.05,0.05)--(1.05,0.95)--(1.95,0.95)--(1.95,0.05)--(1.05,0.05);
\draw[dashed,-] (0.05,1.05)--(0.05,1.95)--(0.95,1.95)--(0.95,1.05)--(0.05,1.05);
\foreach \x in {0,1,2} {
\draw[thick,-,blue] (\x,0.05)--(\x,0.95);
\draw[thick,-,OliveGreen] (\x,1.05)--(\x,1.95);
\draw[thick,-,red] (0.05,\x)--(0.95,\x);
\draw[thick,-,Purple] (1.05,\x)--(1.95,\x);
};
\node at (1.6,0.6) {\scriptsize $(2|1)$};
\node at (0.4,1.6) {\scriptsize $(1|2)$};
\node at (0.3,0.7) {\scriptsize $(12|)$};
\node at (0.7,0.3) {\scriptsize $(21|)$};
\node at (1.3,1.7) {\scriptsize $(|12)$};
\node at (1.7,1.3) {\scriptsize $(|21)$};
\end{scope}
\begin{scope}[scale=1.5, shift={(2.7,0)}]
\node at (-.3,1) {$\longmapsto$};
\draw[draw=none,fill=gray!30] (0.05,1.05) rectangle (1.95,1.95);
\draw[dashed,-] (0.1,0.05) -- (0.95,.05) -- (0.95,0.9)-- (0.1,0.05);
\draw[dashed,-] (0.05,0.1) -- (0.05,0.95) -- (0.9,0.95) -- (0.05,0.1);
\draw[dashed,-] (1.1,1.05) -- (1.95,1.05) -- (1.95,1.9)-- (1.1,1.05);
\draw[dashed,-] (1.05,1.1) -- (1.05,1.95) -- (1.9,1.95) -- (1.05,1.1);
\draw[dashed,-] (1.05,0.05)--(1.05,0.95)--(1.95,0.95)--(1.95,0.05)--(1.05,0.05);
\draw[dashed,-] (0.05,1.05)--(0.05,1.95)--(0.95,1.95)--(0.95,1.05)--(0.05,1.05);
\foreach \x in {0,1,2} {
\draw[thick,-,blue] (\x,0.05)--(\x,0.95);
\draw[thick,-,OliveGreen] (\x,1.05)--(\x,1.95);
\draw[thick,-,red] (0.05,\x)--(0.95,\x);
\draw[thick,-,Purple] (1.05,\x)--(1.95,\x);
};
\node at (1.6,0.6) {\scriptsize $(2|1)$};
\node at (0.4,1.6) {\scriptsize $(1|2)$};
\node at (0.3,0.7) {\scriptsize $(12|)$};
\node at (0.7,0.3) {\scriptsize $(21|)$};
\node at (1.3,1.7) {\scriptsize $(|12)$};
\node at (1.7,1.3) {\scriptsize $(|21)$};
\end{scope}
\begin{scope}[scale=1.5, shift={(6,0)}]
\draw[dashed,-] (0.1,0.05) -- (0.95,.05) -- (0.95,0.9)-- (0.1,0.05);
\draw[dashed,fill=gray!30] (0.05,0.1) -- (0.05,0.95) -- (0.9,0.95) -- (0.05,0.1);
\draw[dashed,-] (1.1,1.05) -- (1.95,1.05) -- (1.95,1.9)-- (1.1,1.05);
\draw[dashed,-] (1.05,1.1) -- (1.05,1.95) -- (1.9,1.95) -- (1.05,1.1);
\draw[dashed,-] (1.05,0.05)--(1.05,0.95)--(1.95,0.95)--(1.95,0.05)--(1.05,0.05);
\draw[dashed,-] (0.05,1.05)--(0.05,1.95)--(0.95,1.95)--(0.95,1.05)--(0.05,1.05);
\foreach \x in {0,1,2} {
\draw[thick,-,blue] (\x,0.05)--(\x,0.95);
\draw[thick,-,OliveGreen] (\x,1.05)--(\x,1.95);
\draw[thick,-,red] (0.05,\x)--(0.95,\x);
\draw[thick,-,Purple] (1.05,\x)--(1.95,\x);
};
\node at (1.6,0.6) {\scriptsize $(2|1)$};
\node at (0.4,1.6) {\scriptsize $(1|2)$};
\node at (0.3,0.7) {\scriptsize $(12|)$};
\node at (0.7,0.3) {\scriptsize $(21|)$};
\node at (1.3,1.7) {\scriptsize $(|12)$};
\node at (1.7,1.3) {\scriptsize $(|21)$};
\end{scope}
\begin{scope}[scale=1.5, shift={(8.7,0)}]
\node at (-.3,1) {$\longmapsto$};
\draw[draw=none,fill=gray!30] (0.05,0.1) -- (0.05,1.95) -- (1.9,1.95);
\draw[dashed,-] (0.1,0.05) -- (0.95,.05) -- (0.95,0.9)-- (0.1,0.05);
\draw[dashed,-] (0.05,0.1) -- (0.05,0.95) -- (0.9,0.95) -- (0.05,0.1);
\draw[dashed,-] (1.1,1.05) -- (1.95,1.05) -- (1.95,1.9)-- (1.1,1.05);
\draw[dashed,-] (1.05,1.1) -- (1.05,1.95) -- (1.9,1.95) -- (1.05,1.1);
\draw[dashed,-] (1.05,0.05)--(1.05,0.95)--(1.95,0.95)--(1.95,0.05)--(1.05,0.05);
\draw[dashed,-] (0.05,1.05)--(0.05,1.95)--(0.95,1.95)--(0.95,1.05)--(0.05,1.05);
\foreach \x in {0,1,2} {
\draw[thick,-,blue] (\x,0.05)--(\x,0.95);
\draw[thick,-,OliveGreen] (\x,1.05)--(\x,1.95);
\draw[thick,-,red] (0.05,\x)--(0.95,\x);
\draw[thick,-,Purple] (1.05,\x)--(1.95,\x);
};
\node at (1.6,0.6) {\scriptsize $(2|1)$};
\node at (0.4,1.6) {\scriptsize $(1|2)$};
\node at (0.3,0.7) {\scriptsize $(12|)$};
\node at (0.7,0.3) {\scriptsize $(21|)$};
\node at (1.3,1.7) {\scriptsize $(|12)$};
\node at (1.7,1.3) {\scriptsize $(|21)$};
\end{scope}
\end{tikzpicture}
\caption{The transvection operations $T_{12}$ on the cell $(1|2)$, on the left, and $(12|)$ on the right, of $\Conf_2(R_2)^+$, as in Example \ref{ex:transvection}.}
\label{fig:transvection}
\end{figure}

\end{exmp}

\begin{rmk}\label{rmk:jordan blocks}
As mentioned before Lemma \ref{lem:out acting on chains}, 
the chain operators given above do not satisfy the relations between $f_i,s_i$ and $t_{12}$ in $\Out(F_g)$. 
For example, we have $(f_2t_{12})^2 = 1$, whereas the transvection operation $(F_2T_{12})^2$ on $R_g$ is not the identity map on the chain level. 

Another class of finite order elements playing a role in what follows are elements in $\Out(F_g)$ coming from isometries of genus $g$ graphs.  Since these have finite order, the action they induce on homology is indeed diagonalizable over $\overline \Q$.  Had these elements acted on the cellular chains with finite order, their action would also be diagonalizable. But we have encountered examples in which such operators have non-trivial Jordan blocks, e.g., the order $4$ rotation of the complete graph $K_4$.
\end{rmk}

\subsection{Separating into irreducibles} \label{sec:irreps}
The free resolution of $\widetilde{H}_*(\Conf_n(R_g)^+;\Q)$ as an $S_n$-representation opens the door to splitting up the calculation into the distinct irreducibles of $S_n$ when working rationally. 
This approach drastically reduces the size of the vector spaces involved, and allows for efficient extraction of specific irreducible multiplicities. Efficiency is particularly important, seeing that the vector spaces in the resolution of Lemma \ref{lem:resolution} have dimension $\sim n^{g-1}\cdot n!$.

Consider any associative ring $R$ and a morphism of free (left) $R$-modules $\psi: R^N\to R^M$. Representing elements of $R^N$ by row vectors, $\psi$ is uniquely represented  
by a matrix $A\in M_{N\times M}(R)$, which  acts on $R^N$ by \emph{right} multiplication.

Specializing this to the group ring $R = \Z[S_n]$, 
the differential $\partial: \Z[S_n]^N \to \Z[S_n]^M$ from Lemma \ref{lem:boundary} is represented by a matrix we shall denote $A_\partial$. We emphasize that the entries of $A_\partial$ are elements in $\Z[S_n]$, characterized in \eqref{eq:differential}. The underlying $\Z$-linear map would in principle be represented by a matrix that is $n!$ times bigger, but we will never use this larger matrix directly.
  The action of generators of $\Out(F_g)$ on this complex is similarly described as $\Z[S_n]$-valued matrices as determined by  Lemma \ref{lem:out acting on chains}, 
and the identification of cellular chains and elements in $\Z[S_n]$ is given in \eqref{eq:identifying cells with permutations}.

Now extend scalars to $\Q$. Lemma \ref{lem:splitting irreducibles} below records the general statement that allows one to split the homology calculations into isotypic components, where all matrices involved are substantially smaller than the original $A_\partial$. The only computational input needed is a realization of the irreducible representations of $S_n$ as explicit matrices, which has already been implemented in Sage \cite{sagemath}.

We recall the notion of multiplicity space. Let $\gp$ be a finite group and $\rho\col \gp\to \End_\C(V_\rho)$ a complex irreducible representation. For any complex $\gp$-representation $W$, the \emph{multiplicity space} of $\rho$ in $W$ is $W^{(\rho)} := W\otimes_\gp V_\rho^*$, where $V_\rho^*$ is the dual representation to $V_\rho$.  More generally, for a $\Z[\gp]$-module $W$, define $W^{(\rho)}$ to be the $\rho$-multiplicity space for the extension of scalars $W_{\C} := W\otimes \C$.
Given a set $\hat{\gp}$ of representatives of the isomorphism classes of irreducible complex $\gp$-representations, Schur's lemma gives a natural isomorphism
\[
W_{\C} \cong \bigoplus_{\tau\in \hat{\gp}} W^{(\tau)} \otimes_{\C} V_\tau.
\]
In particular, $\dim W^{(\tau)}$ is the  number of times $V_\tau$ occurs in $W_{\C}$, and 
any map of $\Q[\gp]$-representations $W\to U$ is uniquely determined by respective maps $W^{(\tau)} \to U^{(\tau)}$ for $\tau\in \hat{\gp}$.

\begin{lem} \label{lem:splitting irreducibles}
Let $\gp$ be a finite group, $\rho\col \gp\to \End_\C(V_\rho)$ a complex irreducible representation. Given a complex of regular $\gp$-representations \[C_\bullet = (\ldots \to \Z[\gp]^{n_i} \xrightarrow{\ \partial_i\ } 
\Z[\gp]^{n_{i-1}} \to \ldots),\] there is an isomorphism, natural in all $\gp$-equivariant maps of complexes,
\begin{equation}\label{eq:multiplicity in homology}
H_i(C_\bullet)^{(\rho^*)} \cong H_i\left( \ldots \to V_\rho^{n_i} 
\xrightarrow{\ \rho[\partial_i]\ }
V_\rho^{n_{i-1}} \to \ldots \right)
\end{equation}
where $\rho^*$ is the dual representation to $\rho$ and $\rho[\partial_i]\in M_{n_i\times n_{i-1}}(\operatorname{End}_\C(V_\rho))$ is the operator $V_\rho^{n_i} \to V_\rho^{n_{i-1}}$ obtained by applying $\rho$ entry-wise to $A_{\partial_i}\in M_{n_i\times n_{i-1}}(\Z[\gp])$.

In particular, the dimensions of the homology on the right hand side of \eqref{eq:multiplicity in homology} are the multiplicity with which $\rho^*$ occurs in $H_*(C_\bullet)$.
\end{lem}

\begin{proof}
Working with complex representations of a finite group, every representation splits as a sum of irreducibles. In particular, the tensor $(-)\otimes_{\gp} V_\rho$ is an exact functor and commutes with taking homology. But since the action $\rho$ gives an isomorphism $\Z[\gp]\otimes_{\gp} V_\rho \cong V_\rho$, we have a natural isomorphism of chain complexes,
\[
 \left(\ldots \to \Z[\gp]^{n_i}
 \xrightarrow{\ A_\partial\ }
 \Z[\gp]^{n_{i-1}} \to \ldots\right) \otimes_{\gp} V_\rho \quad \cong \quad \left( \ldots \to V_\rho^{n_i} 
 \xrightarrow{\ \rho[A_\partial]\ }
 V_\rho^{n_{i-1}} \to \ldots \right).
\]
Passing to the homology of these complexes proves the claim.
\end{proof}

Working with $\gp = S_n$, the formula \eqref{eq:multiplicity in homology} simplifies due to the fact that every $S_n$-representation is self-dual, i.e., $\rho^* \cong \rho$. Moreover, since all $S_n$-characters are defined over $\Q$, the same discussion applies already for rational rather than complex representations.
\begin{cor}
    There are isomorphisms, natural in all continuous self-maps of $R_g$,
    \begin{equation}
        \widetilde{H}_{n-1}(\Conf_n(R_g)^+;\Q)^{(\rho)} \cong \ker(\rho[A_\partial]) \quad \text{and} \quad
        \widetilde{H}_{n}(\Conf_n(R_g)^+;\Q)^{(\rho)} \cong \coker(\rho[A_\partial])        
    \end{equation}
    where $A_\partial$ is the $\Z[S_n]$-valued matrix representing $\Z[S_n]^{\binom{n+g-1}{g-1}}\overset{\partial}\to \Z[S_n]^{\binom{n+g-2}{g-1}}$ from Lemma \ref{lem:boundary}.
\end{cor}

\begin{rmk}Since $M_{N\times M}(\End_\C(V_\rho))\cong M_{Nd\times Md}(\C)$ for $d=\dim(V_\rho)$, the resulting calculation of the (co)kernel is reduced from involving $Nn!\times Mn!$ matrices to $Nd\times Md$ ones. 
For $S_n$, this reduces the matrix sizes by a factor of at least $\sqrt{n!}$ (see \cite{McKay}), e.g. for $S_{10}$ the largest irreducible has dimension $d=768$ compared to $10!\sim 3.6\times 10^6$.
\end{rmk}

\begin{cor}[\textbf{Sign representations}]
The $\Q_{\sgn}$-isotypic component of  $\widetilde{H}_k(\Conf_n(R_g)^+;\Q)$ has multiplicity $\binom{k+g-1}{g-1}$ for $k=n-1$ and $n$, and has multiplicity $0$ otherwise. 
Explicitly, every cell $(\sigma,\chi)$ gives a cycle $\sum_{\tau\in S_n} \sgn(\tau)\tau\cdot(\sigma,\chi)$, and different $S_n$-orbits of those are non-homologous.
\end{cor}
Geometrically, these $\sgn$-isotypic cycles are represented by the loci of all configurations with specified numbers of points on each arc.
\begin{proof}
Lemma \ref{lem:boundary} gives a formula for the cellular boundary $\partial$ of $\Conf_n(R_g)^+$, and by Lemma \ref{lem:splitting irreducibles} the $\rho$-multiplicity space of $\widetilde{H}_k(\Conf_n(R_g)^+;\Q)$ for $k=n-1$ (and $n$) is computed by the cokernel (and kernel) of the linear operator $\rho[\partial]$.

As in \eqref{eq:identifying cells with permutations}, a cell $(\sigma_1\ldots|\ldots|\ldots \sigma_{n-1})$ corresponds to $\sgn(\sigma)\sigma^{-1}\in \Z[S_n]$, hence applying $\rho$ to such a cell results in the endomorphism $\sgn(\sigma)\rho(\sigma)^{-1}\in \operatorname{End}_\Q(V_\rho)$. In particular, when $(\rho = \sgn)$ every cell is sent by $\rho$ to $+1\in \End_\Q(\Q_\sgn)$, and \eqref{eq:differential} immediately degenerates to $\rho[\partial]=0$. We conclude that the $\sgn$-multiplicity space of the homology is isomorphic to that of the cellular chains, which is simply $\Q^{\binom{k+g-1}{g-1}}$. It further follows that $\partial$ restricts to $0$ on the $\sgn$-isotypic component of $C_k^{\operatorname{CW}}(\Conf_n(R_g)^+)$. Recalling that the projection onto the $\sgn$-isotypic component is given by anti-symmetrization, the claim follows.
\end{proof}

\section{From graph configuration space to tropical moduli space}

We now briefly recall the definition of the tropical moduli space $\Delta_{g,n}$ and establish a connection between $\Delta_{2,n}$ and a particular graph configuration space. We will then use the techniques of \S\ref{sec:homology of conf} to compute the homology of $\Delta_{2,n}$.

A {\em tropical curve} is a vertex-decorated metric graph. More precisely, it is a tuple of data $(G,w,m,l)$ where $G$ is a connected graph, possibly with loops and parallel edges, $w\colon V \to \Z_{\geq 0}$ a {\em weight} function on the set $V$ of vertices, $m\colon \{1,\dots,n\} \to V$ a {\em marking} function, and $l\colon E \to \R_{> 0}$ an {\em edge-length} function. These data must satisfy the following {\em stability} condition: for each $v\in V$, we require $2w(v) + \mathrm{val}(v) + |m^{-1}(v)| > 2$, where $\mathrm{val}$ is the graph theoretical valence. The genus of a tropical curve is $|E(G)| - |V(G)| + 1 + \sum_{v\in V(G)} w(v)$. Let $\Delta_{g,n}$ denote the {\em moduli space of genus $g$, $n$-marked tropical curves}.  This is a topological space that parametrizes isomorphism classes of tropical curves of genus $g$ and $n$ markings having total edge length 1. It is glued from quotients of the standard simplices  inside $\R_{\geq 0}^{E(G)}$  for graphs $G$, thus inheriting the quotient topology. For a formal definition, see \cite{CGP1}. 

A {\em bridge} in a connected graph is an edge whose deletion disconnects the graph. 
The {\em bridge locus}, denoted $\Delta^\br_{g,n} \subset \Delta_{g,n}$, is the closure in $\Delta_{g,n}$ of the locus of tropical curves with bridges.

Now let $g=2$.  Recall the graph $\Theta$, 
now regarded as a metric graph with two vertices $v_1,v_2$ and three edges $e_1,e_2,e_3$ between them of equal lengths. 
We say a tropical curve $(G,w,m,l)\in \Delta_{2,n}$ has {\em theta type} if $G$ is homeomorphic to $\Theta$ and its marking function $m$ is injective.

\begin{lem}\label{lem:homeo}
Let $\iso(\Theta)$ be the group of isometries of $\Theta$. We have a homeomorphism of topological spaces \[ ((\Delta^2)^\circ \times \Conf_n(\Theta))/\mathrm{Iso}(\Theta) \; \cong \; \Delta_{2,n}\setminus\Delta_{2,n}^{\mathrm{br}},\]
where $(\sigma,\tau)\in S_2\times S_3 \cong \iso(\Theta)$  acts on $(\Delta^2)^\circ$ through the permutation action of $\tau$ on $\R^3$ and on $\Conf_n(\Theta)$ through the natural action of $\iso(\Theta)$ on $\Theta$.
\end{lem}

\begin{proof}
Let $(\Delta^2)^\circ$ denote the interior of the standard 2-simplex. There is a continuous map \[f: (\Delta^2)^\circ \times \Conf_n(\Theta) \to \Delta_{2,n}\setminus \Delta_{2,n}^\br\] given as follows. Let $X$ be a configuration of $n$ points on $\Theta$ and $(r_1,r_2,r_3) \in (\Delta^2)^\circ$. Then $f((r_1,r_2,r_3),X)$ is the isomorphism class of the following tropical curve $(G,w,m,l)$ of theta type. The graph $G$ is obtained from $\Theta$ by subdividing each edge at every point in the configuration. The marking function $m$ is set to have $m(i)$ be the vertex at point $i$ in the configuration $X$. The length function $l$ is obtained by scaling the 1-cells $e_1,e_2$, and $e_3$ to have lengths $r_1, r_2$, and $r_3$, respectively.

By \cite[Lemma 3.1]{chan2015topology}, a tropical curve in $\Delta_{2,n}$ has theta type if and only if it lies in $\Delta_{2,n}\setminus \Delta_{2,n}^\br$. Therefore $f$ is surjective.
Moreover, two elements in $(\Delta^2)^\circ \times \Conf_n(\Theta)$ have the same image if and only if they are in the same orbit under the action of $\iso(\Theta)$. So $f$ descends to a homeomorphism from the quotient space $((\Delta^2)^\circ \times \Conf_n(\Theta))/\iso(\Theta)$ to $\Delta_{2,n}\setminus \Delta_{2,n}^\br$.
\end{proof}

\cite[Theorem 1.1]{cgp-marked} establishes that $\Delta_{2,n}^\br$ is contractible. Therefore Lemma~\ref{lem:homeo} enables us to relate the reduced rational cohomology of $\Delta_{2,n}$ with that of $\Conf_n(\Theta)$.   

\begin{thm} \label{thm: red cohom of Delta2n and red cohom of Conf}
There is an $S_n$-equivariant homotopy equivalence
\begin{equation}\Delta_{2,n} \simeq (S^2\wedge \Conf_n(\Theta)^+)/\mathrm{Iso}(\Theta),
\end{equation}
where $\wedge$ is the smash product and $\iso(\Theta)\cong S_2\times S_3$ acts on the sphere $S^2$ by reversing orientation according to the sign of the permutation in $S_3$. 

In particular, there is an isomorphism of $S_n$-representations \begin{equation}
    \widetilde{H}^i(\Delta_{2,n};\Q) \cong (\sgn_3 \otimes \widetilde{H}^{i-2}(\Conf_n(\Theta)^+;\Q))^{\mathrm{Iso}(\Theta)},
\end{equation}
where $\sgn_3$ is the sign representation of $S_3$ in $\mathrm{Iso}(\Theta) \cong S_2\times S_3$,
and the superscript denotes the $\mathrm{Iso}(\Theta)$-invariant part. 
Similarly, there is an equivariant isomorphism \begin{equation}
    \widetilde{H}_i(\Delta_{2,n};\Q) \cong (\sgn_3 \otimes \widetilde{H}_{i-2}(\Conf_n(\Theta)^+;\Q))_{\mathrm{Iso}(\Theta)},
\end{equation}
where the subscript $\iso(\Theta)$ denotes the coinvariant quotient.
\end{thm}

\begin{proof}
By Lemma~\ref{lem:homeo}, we have \[ \Delta_{2,n}\setminus\Delta_{2,n}^{\mathrm{br}} \; \cong \; ((\Delta^2)^\circ \times \Conf_n(\Theta))/\mathrm{Iso}(\Theta).\] So their one-point compactifications are homeomorphic: \begin{equation} (\Delta_{2,n}\setminus\Delta_{2,n}^{\mathrm{br}})^+ \cong (((\Delta^2)^\circ \times \Conf_n(\Theta))/\mathrm{Iso}(\Theta))^+. \label{eqn: homeomorphic one-point compactification}
\end{equation} Since the bridge locus is contractible \cite[Theorem 1.1]{cgp-marked}, the left-hand side of \eqref{eqn: homeomorphic one-point compactification} is homotopy equivalent to $\Delta_{2,n}$.
The right-hand side of \eqref{eqn: homeomorphic one-point compactification} is homeomorphic to
the space $((\Delta^2)^\circ \times \Conf_n(\Theta))^+/\mathrm{Iso}(\Theta)$, where $\iso(\Theta)$ acts trivially on the point $\infty$. Then the first claim follows from the identification $(X \times Y )^+ = X^+ \wedge Y^+
$, along with the fact that $((\Delta^2)^\circ)^+\cong S^2$.

Passing to rational cohomology, we deduce 
\[ \widetilde{H}^i(\Delta_{2,n};\Q) \cong
\widetilde{H}^i((S^2\wedge \Conf_n(\Theta)^+)/\mathrm{Iso}(\Theta);\Q) \cong \widetilde{H}^i((S^2\wedge \Conf_n(\Theta)^+);\Q)^{\iso(\Theta)}.\]  
By the K\"unneth formula,
\[\widetilde{H}^*(S^2 \wedge \Conf_n(\Theta)^+;\Q) \cong \widetilde{H}^*(S^2;\Q) \otimes \widetilde{H}^*(\Conf_n(\Theta)^+;\Q).\] 
Since the reduced cohomology of $S^2$ is supported in degree 2, where it is $1$-dimensional, and $\iso(\Theta)$ acts through the orientation reversing action of $S_3$,
it follows that $\widetilde{H}^2(S^2)$ is $\iso(\Theta)$-equivariantly isomorphic to $\mathrm{triv}_2 \otimes \mathrm{sgn}_3$. We obtain the desired isomorphism of rational vector spaces, and every identification above is equivariant with respect to the $S_n$-actions induced by permuting marked points. 
\end{proof}

\subsection{Isometries of the graph $\Theta$.} \label{sec:isometries for theta} 
 The last ingredient needed to compute the homology of $\Delta_{2,n}$ using Theorem \ref{thm: red cohom of Delta2n and red cohom of Conf} and the techniques of \S\ref{sec:homology of conf} is the action of the graph automorphism group $\iso(\Theta)$ of the Theta graph $\Theta$ on $\widetilde{H}_*(\Conf_n(\Theta)^+;\Q)$.
 For computations, we choose the particular homotopy
 equivalence $\Theta\xrightarrow{\ \sim\ } R_2$ that collapes the edge $e_3$ and sends $e_i$ to the $i$-th arc in $R_2$ for $i=1,2$.  This map induces a homotopy equivalence on configuration spaces $\Conf_n(\Theta)^+\xrightarrow{\ \sim\ } \Conf_n(R_2)^+$, and a group homomorphism
 $\iso(\Theta) \to \Out(F_2)$ as in Proposition~\ref{prop:iso(G) acting through Out}.  We need only consider a generating set of $\iso(\Theta)$, for example:
\begin{itemize}
    \item the order 6 isomorphism, exchanging the vertices and permuting the edges $e_1, e_2,$ and $e_3$ in a $3$-cycle; and
    \item the top swap $t$, fixing the vertices and exchanging $e_1$ and $e_2$.
\end{itemize}
Finally, Lemma \ref{lem:out acting on chains} then gives formulas for the $\iso(\Theta)$-action on cellular chains, and 
 Lemma \ref{lem:splitting irreducibles} lets one calculate the multiplicity space of an individual  irreducible representation $\rho$.

\subsection{Tabulation of data} \label{sec:tabulation}
The above calculation was implemented in Sage \cite{sagemath}, and the resulting irreducible decompositions of the codimension $1$ homology $H_{n+1}(\Delta_{2,n};\Q)$ are shown in Tables~\ref{table:data} and~\ref{table:partialdata}. In these tables, for every partition $\lambda \vdash n$,  $\Specht_{\lambda}$ denotes the Specht module corresponding to $\lambda$, and they are written in reverse lexicographic ordering of partitions. 

Using the formula \cite{CFGP} for the equivariant Euler characteristic of $\Delta_{2,n}$ and the fact that the homology is concentrated only in degrees $n+1$ and $n+2$,  
knowing $H_{n+1}(\Delta_{2,n};\Q)$ is equivalent to knowing $H_{n+2}(\Delta_{2,n};\Q)$.
Please visit \href{https://github.com/ClaudiaHeYun/BCGY}{this URL}\footnote{\texttt{https://github.com/ClaudiaHeYun/BCGY}} for the code we used as well as a web application that presents the data in other ways, including
\begin{itemize}
    \item Frobenius characteristic of codimension $1$ homology $H_{n+1}(\Delta_{2,n};\Q)$ for $n\leq 10$;
    \item Frobenius characteristic of codimension $0$ homology $H_{n+2}(\Delta_{2,n};\Q)$ for $n\leq 10$;
    \item expansions of these symmetric functions in various bases for symmetric functions, e.g., the  elementary symmetric functions;
    \item partial expansions of $H_{n+1}(\Delta_{2,n};\Q)$ and $H_{n+2}(\Delta_{2,n};\Q)$ in the Schur basis for $n\leq 25$.
\end{itemize}

\begin{rmk}
We briefly discuss the performance of our Sage program. The highest $n$ for which we obtain the full homology representation is $n=10$, where the largest irreducible representation has dimension $768$. The matrix used to compute its multiplicity has dimensions $31488 \times 7680$. Computations of irreducible multiplicity for any $n$ never exceeded 24 hours, but computations for large irreducibles with $n\geq 11$ crashed due to insufficient memory.
\end{rmk}

Beyond $n=10$, we were only able to calculate multiplicities of Specht modules of small dimension. Table \ref{table:partialdata} shows partial irreducible decompositions of $H_{n+1}(\Delta_{2,n};\Q)$ for $11\le n \le 17$. The summands are presented as conjugate pairs of partitions, where the set of pairs is ordered reverse-lexicographically. The unknown multiplicities are indicated as $``\textcolor{red}{(? \text{ for }\lambda^*\leq\lambda\leq \lambda_0)}"$, indexed by all partitions that are lex-larger than their conjugate partition and lex-smaller than $\lambda_0$. Any missing partition outside of the unknown range occurs with multiplicity $0$, and similarly for their conjugate partitions.
\begin{table}
{\small
\begin{tabular}{|r|p{.9\textwidth}|}
\hline
    $n$ & Partial irreducible decomposition (listed as conjugate pairs of partitions) \\
\hline
        12 & $(\Specht_{(12)}) + (\Specht_{(11, 1)}) + (2\Specht_{(10, 2)}) + (3\Specht_{(3, 1^9)}) +(4\Specht_{(9, 3)} + 3\Specht_{(2^3, 1^6)}) +  (8\Specht_{(9, 2, 1)} + 7\Specht_{(3, 2, 1^7)}) + (3\Specht_{(9, 1^3)} + 4\Specht_{(4, 1^8)}) + (7\Specht_{(8, 4)} + 3\Specht_{(2^4, 1^4)}) + (19\Specht_{(8, 3, 1)}) + \textcolor{red}{(? \text{ for } \lambda^*\leq \lambda \leq (8,2^2))}$ \\
\hline        
        13 & $(\Specht_{(13)}+2\Specht_{(1^{13})})+(4\Specht_{(11,2)}+3\Specht_{(2^2,1^9)})+(5\Specht_{(10,3)}+5\Specht_{(2^3,1^7)}) + \textcolor{red}{(? \text{ for } \lambda^*\leq \lambda \leq (10,2,1))}$ \\
\hline        
        14 & $(\Specht_{(12,2)})+(4\Specht_{(12,1^2)})+(5\Specht_{(11,3)}+5\Specht_{(2^3,1^8)})+(5\Specht_{(3,1^{11})}) + \textcolor{red}{(? \text{ for } \lambda^*\leq \lambda \leq (11,2,1))}$ \\
\hline        
        15 & $(2\Specht_{(1^{15})}) + (5\Specht_{(2^2,1^{11})}) + (6\Specht_{(13,1^2)}) + \textcolor{red}{(? \text{ for } \lambda^*\leq \lambda \leq (12,3))}$ \\
\hline        
        16 & $(2\Specht_{(16)})+(\Specht_{(15,1)})+(4\Specht_{(14,2)})+(\Specht_{(14,1,1)} +7\Specht_{(3,1^{13})})+ \textcolor{red}{(? \text{ for } \lambda^*\leq \lambda \leq (13,3))}$ \\
\hline        
        17 & $(\Specht_{(17)}+2\Specht_{(1^{17})})+(8\Specht_{(15,2)}+7\Specht_{(2^2,1^{13})}) + (0\cdot \Specht_{(15,1^2)}) + \textcolor{red}{(? \text{ for } \lambda^*\leq \lambda \leq (14,3))}$ \\
\hline
\end{tabular}
}
\vspace{.2cm}

    \caption{Partial irreducible decomposition of $H_{n+1}(\Delta_{2,n};\Q)$ for $n\leq 17$.}
    \label{table:partialdata}
\end{table}

For $18 \leq n \leq 22$, we obtained multiplicities for $\Specht_{(n)}$, $\Specht_{(1^n)}$, $\Specht_{(n-1,1)}$ and $\Specht_{(2,1^{(n-2)})}$, and for $23 \le n \le 25$, we obtained multiplicities for $\Specht_{(n)}$, $\Specht_{(1^n)}$ only.  
All of the  
multiplicities are consistent with the following  Remark~\ref{rmk:triv-sgn} and Conjecture~\ref{conj:std}.

\begin{rmk}\label{rmk:triv-sgn}
There are explicit formulas for the multiplicities of the trivial and sign representations 
in $\widetilde{H}_*(\Delta_{2,n};\Q)$.
The multiplicity of the sign representation $\Specht_{(1^n)}$ in $H_*(\Delta_{2,n})$ is
\begin{equation*}
    \begin{cases}
    \lfloor \frac{n}{6} \rfloor & n \text{ even} \\
    0 & n \text{ odd}
    \end{cases} \quad\text{in degree $*=n+2$, and}\quad
    \begin{cases}
    0 & n \text{ even} \\
    \lfloor \frac{n}{6} \rfloor & n \text{ odd}
    \end{cases}
    \quad\text{for $*=n+1$.}
\end{equation*}

For the trivial representation $\Specht_{(n)}$, its multiplicity in $H_*(\Delta_{2,n})$ is
\begin{equation*}
    \begin{cases}
    0 & n\equiv 0 \mod 4 \\
    0 & n\equiv 1 \mod 4 \\
    \lfloor \frac{n+10}{12} \rfloor & n\equiv 2 \mod 4 \\
    \lfloor \frac{n+1}{12} \rfloor & n\equiv 3 \mod 4
    \end{cases} \quad\text{for $*=n+2$,}\;
    \begin{cases}
    \lfloor \frac{n+8}{12} \rfloor & n\equiv 0 \mod 4 \\
    \lfloor \frac{n-1}{12} \rfloor & n\equiv 1 \mod 4 \\
    0 & n\equiv 2 \mod 4 \\
    0 & n\equiv 3 \mod 4
    \end{cases}
    \quad\text{for $*=n+1$.}
\end{equation*}
Note that in these cases the multiplicity in $\widetilde{H}_*(\Delta_{2,n})$ is nonzero in exactly one degree $*$, which means they are also completely encoded in the $S_n$-equivariant Euler characteristic of $\Delta_{2,n}$ as computed by Faber (see \cite{CFGP}).

 These formulas were obtained in \cite[Theorems 6.2 and 6.4]{Conant-Costello-Turchin-Weed} who used hairy graph complexes.
 Alternatively, in \cite[Section 4.4]{gadish-hainaut}, it is explained that the calculations in \cite[Corollaries 19.8 and 19.10]{powell2018higher} translate to a complete description of the $\Out(F_g)$-representation on the trivial and sign isotypic components of $\widetilde{H}_*(\Conf_n(R_g)^+;\Q)$, and then \cite[Proposition 1.11]{gadish-hainaut} details how the latter translates to $\widetilde{H}_*(\Delta_{2,n})$.
We also learned through private communication with O. Tommasi that these multiplicities can be computed explicitly using dimensions of spaces of modular forms. 

Another way to derive the multiplicity formula for the sign representation was recently communicated to us by B.~Ward; it involves modular forms, via Lie graph homology.  The work \cite{ward-massey} relates $H_*(\Delta_{g,n};\Q)$ with the Lie graph homology which may be identified with $H_*^{Grp}(\Gamma_{g,n};\Q)$ studied in \cite{hatcher-vogtmann-homology-stability}, \cite{chkv}. (The groups $\Gamma_{g,n}$, which generalize $\Out(F_g) = \Gamma_{g,0}$ and $\Aut(F_g) = \Gamma_{g,1}$, were introduced in \cite{hatcher-homological}).  In forthcoming work, Ward calculates in genus $2$ that
\[\sum_{i\geq 0}\dim H_i(\Delta_{2,n})_{\mathrm{sgn}} = \left \lfloor \tfrac{n-2}{4} \right \rfloor - \dim H_{n+1} (\Gamma_{2,n})_\mathrm{sgn}.  \]
Then \cite[Theorem 3.10]{chkv} implies $\dim H_{n+1} (\Gamma_{2,n})_\mathrm{sgn} = \lfloor \frac{n-2}{4} \rfloor - \lfloor \frac{n}{6} \rfloor$, so the total dimension $\sum \dim H_i(\Delta_{2,n})_\mathrm{sgn}=\lfloor \frac{n}{6}\rfloor.$  Combining this with the knowledge of Euler characteristics \cite{CFGP} and the fact that $H_*(\Delta_{2,n})$ is concentrated in two degrees, the sign multiplicity formula above follows again.
\end{rmk}

Multiplicities of other irreducibles remain mysterious. For $\Specht_{(n-1,1)}$ and $\Specht_{(2,1^{n-2})}$, however, we observe the following pattern, verified computationally for up to $n=22$ marked points. 
\begin{conj}\label{conj:std}
For all $n\ge 2$, in the $S_n$-representation $H_*(\Delta_{2,n};\Q)$,
    the multiplicity of the standard representation $\Specht_{(n-1,1)}$ is \[
    \begin{cases}
    \left\lfloor \frac{n}{4} \right\rfloor & n\equiv 2 \mod 4\\
    \left\lfloor \frac{n+2}{6} \right\rfloor & n\equiv 1 \mod 4\\
    0 & \text{otherwise.}
    \end{cases}
\quad\text{for $*=n+2$,}\quad
    \begin{cases} \left\lfloor \frac{n}{12} \right\rfloor & n \equiv 0 \mod 4\\
    0 & \text{otherwise.}
    \end{cases}
    \quad\text{for $*=n+1$},
    \] 
and the multiplicity of $\Specht_{(2,1^{n-2})} \cong \sgn\otimes \Specht_{(n-1,1)}$ is 
   \[
   \begin{cases}
    0 & n \text{ is odd}\\
    \left\lfloor \frac{n+4}{6} \right\rfloor & n \text{ is even}.
    \end{cases}
   \quad \text{ for } *=n+2, 
   \text{ always } 0 \text{ for } *=n+1.\]
\end{conj}

Conjecture~\ref{conj:std} was resolved a few months after a preprint of this paper appeared.  The third author observed with Hainaut in \cite[Example 6.9]{gadish-hainaut} that these multiplicities follow from the work of Powell--Vespa \cite{powell2018higher} and the more recent work of Powell \cite{powell}. An interpretation of these multiplicities in terms of modular forms would be pleasing.


\end{document}